\documentclass{amsart}

\usepackage{amsfonts, amsmath, amssymb}
\usepackage{amscd}
\usepackage {graphicx}
\usepackage{graphics}
\def\timenow{%
  \@tempcnta=\time \divide\@tempcnta by 60 \number\@tempcnta:\multiply
  \@tempcnta by 60 \@tempcntb=\time \advance\@tempcntb by -\@tempcnta
  \ifnum\@tempcntb <10 0\number\@tempcntb\else\number\@tempcntb\fi}

\def\TODAY{\number\year-\ifcase\month\or 01\or 02\or 03\or 04\or 05\or
    06\or 07\or 08\or 09\or 10\or 11\or 12\fi-\number\day}

\theoremstyle{plain}
\newtheorem{theorem}{Theorem}[section]

\newtheorem{lemma}[theorem]{Lemma}

\newtheorem{proposition}[theorem]{Proposition}

\theoremstyle{definition}
\newtheorem{definition}[theorem]{Definition}
\newtheorem{remark}[theorem]{Remark}
\newtheorem{example}[theorem]{Example}

\numberwithin{equation}{section}

\newcommand{\baseRing}[1]{\ensuremath{\mathbb{#1}}}
\newcommand{\Z}{\baseRing{Z}}
\newcommand{\R}{\baseRing{R}}
\newcommand{\C}{\baseRing{C}}
\newcommand{\N}{\baseRing{N}}
\newcommand{\Q}{\baseRing{Q}}

\def\pd#1{ \partial_{#1} }
\newcommand{\LL}{{\mathcal L}}

\newcommand{\UU}{{\mathcal U}}
\newcommand{\NN}{{\mathcal N}}
\newcommand{\CC}{{\mathcal C}}
\newcommand{\EE}{{\mathcal E}}

\DeclareMathOperator{\GL}{GL}

\def\cprime{$'$}

\begin{document}
\title[Bivariate rational hyper\-geo\-metric functions]{The structure of
bivariate rational hyper\-geo\-metric functions}

\author[E.~Cattani]{Eduardo Cattani}
\address[EC]{Department of Mathematics and Statistics \\
University of Massachusetts at Amherst \\
Amherst, MA 01003, USA}
\email{cattani@math.umass.edu}

\author[A.~Dickenstein]{Alicia Dickenstein}
\address[AD]{Departamento de Matem\'atica\\
FCEN, Universidad de Buenos Aires \\
(1428) Bue\-nos Aires, Argentina.}
\email{alidick@dm.uba.ar}

\author[F.~Rodr\'{\i}guez Villegas]{Fernando Rodr\'{\i}guez Villegas}
\address[FRV]{Department of Mathematics \\
University of Texas at Austin\\
Austin, TX 78712-1082, USA}
\email{villegas@math.utexas.edu}

\begin{abstract}
  We describe the structure of all codimension-two lattice
  configurations $A$ which admit a stable rational
  $A$-hyper\-geo\-me\-tric function, that is a rational function $F$
  all whose partial derivatives are non zero, and which is a solution
  of the $A$-hyper\-geo\-metric system of partial differential
  equations defined by Gel'fand, Kapranov and Zelevinsky.  We show,
  moreover, that all stable rational $A$-hyper\-geo\-me\-tric
  functions may be described by toric residues and apply our results
  to study the rationality of bivariate series whose coefficients are
  quotients of factorials of linear forms.
\end{abstract}
\maketitle
\section{Introduction}

Let $A = \{a_1,\dots,a_{n}\} \subset \Z^d$, be a configuration of
lattice points spanning $\Z^d$.  We also denote by $A$ the ${d \times
  n}$ integer matrix with columns $a_1, \dots, a_{n}$.  We say that
the configuration $A$ is {\em regular} if the points of $A$ lie in a
hyperplane off the origin.  The {\it dimension} of $A$ is defined as
the dimension of the affine span of its columns and the {\it
  codimension} as the rank of the lattice
\begin{equation}\label{lattice}
 M :=  \ \{v\in \Z^{n} : A\cdot v =   0\}.
\end{equation}

Following Gel\cprime fand, Kapranov and Zelevinsky \cite{gkz89,gkz90}
we associate to $A$ and a {\em parameter}  vector $\beta \in \C^d$ a left ideal in the
Weyl algebra in $n$ variables $D_n:=\C\langle z_1,\dots,z_n,\pd
1,\dots,\pd n\rangle$ as follows.

\begin{definition}\label{def:hypergeom}
 Given $A\in\Z^{d\times n}$ of rank $d$ and a vector $\beta \in
 \C^d$, the {\it $A$-hyper\-geo\-metric system} with parameter $\beta$ is
 the left ideal $H_A(\beta)$ in the Weyl algebra $ D_n$ generated by
 the {\sl toric operators} $\ \partial^u - \partial^v$, for all $
 u,v\in \N^n$ such that $u - v \in M$, and the {\sl Euler
  operators} $ \,\sum_{j= 1}^n a_{ij} z_j \partial_j - \beta_i \,$
 for $\, i= 1,\dots,d $.
 A holomorphic function $F(z_1,\dots,z_n)$, defined in some open set
 $U\subset \C^n$, is said to be {\it $A$-hyper\-geo\-metric of degree
  $\beta$} if it is annihilated by $H_A(\beta)$.
\end{definition}

$A$-hyper\-geo\-metric systems include as special cases the homogeneous
versions of classical hyper\-geo\-metric systems
in $n-d$ variables.
The ideal $H_A(\beta)$ is always holonomic and if $A$ is regular it
has regular singularities.  The singular locus of the hyper\-geo\-metric
$D_n$-module $D_n/H_A(\beta)$ equals the zero locus of the principal
$A$-determinant $E_A$,  whose irreducible factors are  the sparse discriminants
$D_{A'}$ corresponding to the facial subsets $A'$ of $A$ \cite{gkz89,gkzbook}.

Often, the existence of special solutions to a system of equations
imposes additional structure on the data (see for example a recent
preprint \cite{Beukers} of Beukers on algebraic A-hyper\-geo\-metric
functions.)  In this paper we are interested in the constraints
imposed on $A$ by the existence of rational $A$-hyper\-geo\-metric
functions.  All $A$-hyper\-geo\-metric systems admit polynomial
solutions for parameters $\beta$ in $\N A$, which are closely related
to the solutions of an integer programming problem associated to the
data $(A,\beta)$ \cite{sst_compo}. Likewise, for every $A$ there exist
Laurent polynomial solutions to the $A$-hyper\-geo\-metric
system. Clearly, these rational solutions are annihilated by a
sufficiently high partial derivative. The goal of this paper is to
characterize all codimension-two lattice configurations $A$ which
admit a rational $A$-hyper\-geo\-metric function none of whose derivatives
vanishes.  Such rational functions are called {\em stable}.

We will assume that $A$ is not a {\em pyramid}; that is, a configuration
all of whose points, except one, are contained in a hyperplane.
This entails no loss of generality.
Indeed, suppose  the subset
$A' = \{a_1,\dots,a_{n-1}\}$ lies in a hyperplane not containing $a_n$,
then all $A$-hyper\-geo\-metric functions are of the form:
\begin{equation}\label{pyramid}
F(z_1,\dots,z_n) = z_n^\gamma \  F'(z_1,\dots,z_{n-1}),
\end{equation}
where $F'$ is $A'$-hyper\-geo\-metric. 
Hence, if $A$ is a pyramid over a configuration $A'$ which  admits a
stable $A'$-hyper\-geo\-metric function then, clearly, so does $A$.

In order to state our results we need to describe certain special 
configurations, which play an important role throughout this paper.
A configuration $A\subset \Z^d$ is said to be
a {\it Cayley configuration\/} if there exist vector configurations
$A_1,\dots,A_{s}$ in $\Z^r$ such that
\begin{equation}
\label{Cayley}
A \,\,\, =   \,\, \,
\{e_1\} \! \times A_1 \! \, \, \cup \,\,
\cdots \, \, \cup \,\,
\{e_{s}\} \! \times \! A_{s}
\,\, \subset \,\, \Z^{s} \times \Z^r ,
\end{equation}
where $e_1,\dots, e_{s}$ is the standard basis of $\Z^{s}$.  Note that
we may assume that all the $A_i$'s consist of at least two points
since, otherwise, $A$ would be a pyramid.

A Cayley configuration is said to be a {\em Lawrence} configuration if all the
configurations $A_i$ consist of exactly two points.  Thus, up to affine isomorphism, we may assume that $A_i = \{0,\gamma_i\}$, $\gamma_i \in \Z^r\backslash \{0\}$.  It follows from our assumptions that the vectors $\gamma_1,\dots,\gamma_s$ must span
$\Z^r$ over $\Z$.   We note that the codimension of a Lawrence configuration is $s-r$.

We say that a Cayley configuration is 
{\em essential} if $s=r+1$ and the Minkowski sum
$\,\sum_{i \in I} A_i \,$ has affine dimension at
least $|I|$ for every proper subset $I$ of $\{1,\dots,r+1\}$.
For a codimension-two essential Cayley configuration,
 $r$ of the configurations $A_i$, say
$A_1,\dots, A_r$, must consist of two vectors and the remaining one,
$A_{r+1}$, must consist of three vectors.  If we set $A_i=\{\mu_i, \nu_i\}\subset \Z^r$, $i=1,\dots, r$, then it follows from the fact that
$A$ is essential that the vectors $\gamma_i = \nu_i - \mu_i$ are linearly independent over $\Q$.  Thus, modulo affine equivalence, we may assume without loss of generality that $A_i = \{0,\gamma_i\}$, $i=1,\dots, r$, where $\gamma_1,\dots, \gamma_r$ are linearly independent over $\Q$ and $A_{r+1} = \{0,\alpha_1,\alpha_2\}$ with $\alpha_1$,
$\alpha_2$ are not both contained in a subspace generated by a proper subset of
$\gamma_1,\dots, \gamma_r$.

In order to simplify our statements we will allow ourselves a slight abuse of
notation and consider the configuration
$$\left(
\begin{array}{cccc}
1 & 1 & 0 & 0\\
0 & 0 & 1 & 1
\end{array}
\right)
$$
as a Lawrence configuration and the zero-dimensional configuration $ (1\  1\  1)$ as a Cayley essential configuration.

It has been shown in \cite{binom, rhf} that both Lawrence configurations and 
essential Cayley configurations admit stable rational $A$-hyper\-geo\-metric functions.  This is done by exhibiting explicit functions constructed as 
toric residues.  Our main result asserts that if $A$ has codimension two then these are the only configurations that admit such functions.

\begin{theorem}\label{th:conj}
A codimension two configuration $A$ admits a
stable rational hyper\-geo\-metric function if and only if
it is affinely equivalent to either
an essential Cayley configuration or a Lawrence configuration.
\end{theorem}

As an immediate corollary to Theorem~\ref{th:conj} we obtain a proof
for the codimension-two case of Conjecture~1.3 in \cite{rhf}.  We
recall that a configuration $A$ is said to be {\em gkz-rational} if
the discriminant $D_A$ is not a monomial and $A$ admits a rational
$A$-hyper\-geo\-metric function with poles along the discriminant locus
$D_A=0$.  Such a function is easily seen to be stable.  Thus, by
Theorem~\ref{th:conj}, $A$ must be either a Lawrence or a Cayley
essential configuration.  But, if ${\rm codim}(A)>1$, the sparse
discriminant of a Lawrence configuration is $1$, and therefore the
only codimension-two gkz-rational configurations are Cayley essential
as asserted by \cite[Conjecture~1.3]{rhf}.

Let us briefly outline the strategy for proving Theorem~\ref{th:conj}.
The fact that an $A$-hyper\-geo\-metric function $F(z)$ of degree
$\beta$ satisfies $d$ independent homogeneity relations, one for each
row of the matrix $A$, implies that the study of codimension-two
rational $A$-hyper\-geo\-metric functions may be reduced to the study
of rational power series in two variables whose coefficients satisfy
certain recurrence relations.  Now, it follows easily from the
one-variable Residue Theorem that the {\it diagonals} of a rational
bivariate power series define algebraic one-variable functions.  On
the other hand, coming from an $A$-hyper\-geo\-metric function, these
univariate functions are classical one-variable hyper\-geo\-metric
functions.  Theorem~\ref{notalgebraic} allows us to reduce the study
of these one-variable functions to those studied by Beukers-Heckman
\cite{bh} (see also \cite{bober,frv}).  Analyzing the possible
functions arising as diagonals of a bivariate rational function leads
us to conclude that $A$ must be affinely equivalent to an essential
Cayley configuration or a Lawrence configuration.

In the latter case, the stable rational $A$-hyper\-geo\-metric
functions have been studied in \cite{binom} where it is shown that an
appropriate derivative of such a function may be represented by a
multivariate residue.  In \S\ref{sec:residues} we show that a similar
result holds for essential Cayley configurations of codimension two.
After recalling the construction of rational $A$-hyper\-geo\-metric
functions by means of toric residues, we show in
Theorem~\ref{th:dim1}, that if the parameter $\beta$ lies in the
so-called Euler-Jacobi cone $\EE$ (see \eqref{EJ}), the space of
rational $A$-hyper\-geo\-metric function of degree $\beta$ is
one-dimensional.  This proves \cite[Conjecture~5.7]{rhf} for any
codimension-two essential Cayley configuration.

Finally, in Section~\ref{sec:rat-bivariate} we apply our results to
study the rationality of classical bivariate hyper\-geo\-metric series (in the sense of Horn,
see Definition~\ref{def:Horn} and Remark~\ref{rem:Horn}).
Theorem~\ref{th:conjclassic} shows that any bivariate Taylor series
whose coefficients are quotients of factorials of integer linear forms
as in \eqref{eq:horn} defines a
rational function only if the linear forms arise from a Lawrence or
Cayley essential configuration. We end up by considering the case
of Horn series supported in the first quadrant.
\medskip

\noindent {\bf Acknowledgments:}
EC would like to thank the Fulbright Program and the University of Buenos
Aires for their support and hospitality.
AD was partially supported by UBACYT X064, CONICET PIP 5617 and ANPCyT
PICT 20569, Argentina.  FRV would like to thank the program RAICES of
Argentina and the NSF for their financial support.  He would also like
to thank the department of Mathematics of the Universidad de Buenos
Aires, Argentina and the Arizona Winter School, where some of this
work was done.

\section{Univariate algebraic hyper\-geo\-metric functions}\label{sec:univariate}

In this section we study algebraic hyper\-geo\-metric
series of the form

\begin{equation}\label{series}
u(z) \ :=\ \sum_{n=0}^\infty \ \frac{\ \prod_{i=1}^r
\,(p_i\,n + k_i)!\ }{\prod_{j=1}^s \,(q_j\,n)!}\ z^n,\quad k_i \in \N.
\end{equation}
We are interested in the case when the series (\ref{series}) has a
finite, non-zero, radius of convergence. Hence we assume that
\begin{equation}
\label{samesum}
{\sum_{i=1}^r p_i \ = \ \sum_{j=1}^s  q_j}.
\end{equation}

\bigskip

The case $k_i=0$ for $i=1,\dots,r$, namely, the series
\begin{equation}\label{centralseries}
v(z) \ :=\ \sum_{n=0}^\infty \ \frac{\ \prod_{i=1}^r
\,(p_i\,n)!\ }{\prod_{j=1}^s \,(q_j\,n)!}\ z^n,\qquad p_i\not=q_j,
\end{equation}
has been studied in \cite{bh, bober, frv}.  If $r=s=0$ all
coefficients are equal to~$1$ and $v(z)=(1-z)^{-1}$ is
rational. Assume then that $r,s>0$.  Using the work of Beukers
and Heckman \cite{bh} it was shown in \cite{frv} that $v$ defines an
algebraic function if and only if the {\em height}, defined as $d := s -
r$, equals~$1$ and the factorial ratios
\begin{equation}
\label{ratios}
A_n \ :=\ \frac{\ \prod_{i=1}^r
\,(p_i\,n)!\ }{\prod_{j=1}^s \,(q_j\,n)!}
\end{equation}
are integral for every $n\in \N$. (In the last case, $v$ is not a
rational function, in fact, since by Stirling the coefficients are, up
to a constant, asymptotic to $1/\sqrt n$ times an exponential.)

Beukers and Heckman \cite{bh} actually gave an explicit classification of
all algebraic univariate hyper\-geo\-metric series. As a consequence, we
can also classify all integral factorial ratio sequences
(\ref{ratios}) of height $1$ (see
\cite[\S~7.2]{frv_book},\cite{vasyunin},
\cite[Theorem~1.2]{bober}). We may clearly assume that
\begin{equation}
\label{gcd}
\gcd(p_1,\dots,p_r,q_1,\dots,q_{r+1}) \ =\ 1.
\end{equation}
Then there exist three infinite families, where $A_n$ is given by
\begin{equation}\label{ratios1}
 \frac{\ ((a+b)\,n)!\ }{(a\,n)!\,(b\,n)!},\qquad \gcd(a,b)
=1,
\end{equation}
\begin{equation}
\label{ratios2}
 \frac{\ (2(a+b)\,n)!\,(b\,n)!\ }{((a+b)\,n)!\,(2b\,n)!\,(a\,n)!},
\qquad \gcd(a,b)=1,
\end{equation}
or
\begin{equation}
\label{ratios3}
\frac{\ (2a\,n)!\,(2b\,n)!\ }{(a\,n)!\,(b\,n)!\,((a+b)\,n)!},
\qquad \gcd(a,b)=1,
\end{equation}
and $52$ sporadic cases listed in \cite[Table~2]{bober}.

\begin{remark}
 Because of the connections with step functions, it is also
 interesting to study integral factorial ratio sequences satisfying
 (\ref {samesum}) but of height different than one.  Partial results
 in this direction are contained in \cite{bell}.  The connections
 with quotient singularities and the Riemann Hypothesis are explored
 in \cite{borisov}.
\end{remark}

Note that we can write a series $u$ as in \eqref{series} as follows
$$
u(z)=\sum_{n\geq 0}h(n)A_n\,z^n,
$$
where $h$ is the polynomial
\begin{equation}\label{fallingpolynomial}
h(x)=\prod_{i=1}^r\prod_{j=1}^{k_i}(p_ix+j)
\end{equation}
and $A_n$ is as in \eqref{ratios}. We now show that $u$ and $v$ can
only be algebraic simultaneously. More generally, we have the following.

\begin{theorem} \label{notalgebraic}
Suppose
$$
u(z):=\sum_{n\geq 0}h(n)A_n\,z^n, \qquad v(z):=\sum_{n\geq 0}A_n\,z^n,
$$
where $h(x)\in \Z[x]$ is non-zero and $A_n$ is as in \eqref{ratios}. Then:

(i) The series $u(z)$ is algebraic  if and only if $v(z)$ is
algebraic. 

(ii) If $u$ is rational then $A_n=1$ for all $n$ and
$$
v(z)=\frac 1{1-z}.
$$
\end{theorem}
\begin{proof}
  We first prove (i).  One direction is clear as
  $u=h(\theta)v$. Suppose then that $u$ is algebraic. By a theorem of
  Eisenstein (see \cite{DP} for a modern treatment and further
  references) the coefficients $h(n)A_n$ of $u$ are integral away from
  a finite set of primes.

 We may assume without loss of generality that $h$ is primitive, i.e.,
 that the $\gcd$ of all of its coefficients is $1$. Hence, for any
 prime $l$ there are at most $\deg(h)$ congruences classes $n\bmod l$
 for which ${\mathfrak v}_l(h(n))>0$, where ${\mathfrak v}_l$ denotes
 the valuation at $l$.

It follows that for all sufficiently large primes $l$ the number of
 exceptions to
\begin{equation}
\label{val1}
{\mathfrak v}_l(A_n)={\mathfrak v}_l(h(n)A_n)\geq 0, \qquad 
0\leq n <l,
\end{equation}
is at most $\deg(h)$, independent of $l$.  In other words, the
valuation at $l$ of the coefficientes of $u$ is essentially that of
the coefficients of $v$. We will exploit this fact in order to prove
the theorem.

It is easy to verify (see \cite{frv1} for  details on the
following discussion) that
$$
{\mathfrak v}_l(A_n)=\sum_{\nu\geq 1} \LL\left(\frac n{l^\nu}\right),
$$
where $\LL$ is the {\it Landau function}
$$
\LL(x):=\sum_{j=1}^s\{q_j x\}-\sum_{i=1}^r\{p_i x\}, \qquad x \in \R.
$$
Here $\{x\}$ denotes the fractional part of $x\in \R$.

The  following properties of $\LL$  hold:
$\LL$ is periodic, with period $1$, locally constant, right continuous
with at most finitely many step discontinuities,
\begin{equation}
\label{L-limit}
\lim_{x\rightarrow 1^-}\LL(x)=d
\end{equation}
and away from the discontinuities
\begin{equation}
\label{L-symmetry}
\LL(-x)=d-\LL(x).
\end{equation}
Furthermore, by a theorem of Landau, $A_n\in \Z$ for all $n$ if and
only if $\LL(x)\geq 0$ for all $x\in \R$.

Since $\LL(0)=0$, for all sufficiently large primes $l$
\begin{equation}
\label{val2}
{\mathfrak v}_l(A_n)=\LL\left(\frac n l\right),\qquad 0\leq n<l.
\end{equation}
Indeed, as $\LL$ is locally constant we have $\LL(x)=0$ for
$x\in[0,\delta_0)$ for some $\delta_0>0$. If $l>\delta_0^{-1}$ and $0\leq n
<l$ then
$$
\frac n{l^k} < \delta_0, \qquad k>1.
$$
 More generally, let
$$
[0,1)=\coprod_\nu \,[\gamma_\nu,\delta_\nu)
$$
be a decomposition of $[0,1)$ into finitely many disjoint subintervals
$I_\nu:=[\gamma_\nu,\delta_\nu)$ such that $\LL$ is constant on each
$I_\nu$. Let $\mu$ the minimum length of the $I_\nu$'s. If
$l>N\mu^{-1}$ for some integer $N>0$ then the number of rationals
of the form $n/l$ in each $I_\nu$ is at least $N$.

Taking $N > \deg h$ and combining \eqref{val1} with
\eqref{val2} we conclude that $\LL(x)\geq 0$ for all $x\in
[0,1)$. Consequently, $A_n\in \Z$ for all $n$ and also $d\geq 0$ by
\eqref{L-limit}.

If $d=0$ then $\LL\equiv 0$ as $\LL(x)\leq d$ by
\eqref{L-symmetry}. It follows that in this case $v(z)=1/(1-z)$ and
$u(z)=h(\theta)v(z)$ are both rational and $r=s=0$. Hence we may
assume $d>0$.

We can write the series $v(z)$ as a hyper\-geo\-metric series (recall we
assume \eqref{samesum})
\begin{eqnarray}
\label{series2}
\nonumber v(z)& \ =\ &   \sum_{n=0}^\infty \ \frac{\prod_{i=1}^r\prod_
{\ell=1}^{p_i} \bigl(\frac{\ell}{p_i}\bigr)_n }
{\prod_{j=1}^s\prod_{\ell=1}^{q_j} \bigl(\frac{ \ell}{q_j}\bigr)_n }\
(z/\kappa )^n\\
& \ =\ &   \sum_{n=0}^\infty \ \frac{\ (\alpha_1)_n\cdots (\alpha_t)_n
\ }
{\ (\beta_1)_n\cdots (\beta_t)_n\ }\ (z/\kappa )^n,
\end{eqnarray}
for some $0<\alpha_i,\beta_j \leq 1$ in $\Q$ for $1\leq i,j\leq t$,
with $\alpha_i \not= \beta_j$ for all $1\leq i,j\leq t$ and where
$\kappa:=\prod_{i=1}^rp_i^{p_i}/\prod_{j=1}^sq_j^{q_j}$. Note that the
number of $\beta$'s that equal~$1$ is precisely $d$. Hence, since
$d\geq 1$ at least one of factors in the denominator of the
coefficient of $(z/\kappa)^n$ is $n!$ and $v(\kappa z)$ is a classical
${}_tF_{t-1}$ hyper\-geo\-metric series.  We remark that the
discontinuities of $\LL$ in $(0,1)$ occur precisely at the
$\alpha_i$'s and $\beta_j$'s.

It follows that $v\in V$, where $V$ is the space of local solutions to
the corresponding hyper\-geo\-metric differential equation $Lv=0$ at some
base point $t_0\neq 0,\kappa,\infty$. The nature of the parameters
$\alpha_i,\beta_j$ of $L$ guarantees that the action of monodromy on
$V$ is irreducible (see \cite[Proposition~3.3]{bh}). On the other
hand, let $U$ be the space of local functions at $t_0$ obtained by
analytic continuation of $u(z)$. The map $h(\theta):V\to U$ preserves
the action of monodromy. By the irreducibility of $V$ this map is
injective. We conclude that the monodromy group of $V$ must be finite
since this is true of $U$ given the hypothesis that $u$ is
algebraic. This shows, in turn, that $v$ is algebraic.

Now assume that $u$ is rational. If $d>0$ the above argument applies
and since the monodromy group of $U$ is trivial so is that of
$V$. Therefore $v$ is rational contradicting the assumption that
$d>0$. To see this note, for example, that $d>0$ implies that $\LL$ is
not identically zero by \eqref{L-limit} and hence $t\geq 1$. In
particular, the local monodromies are not trivial. We conclude that
$d=0$ and consequently, as pointed out above, $\LL=0$ proving (ii).
\end{proof}

\begin{remark}
\label{alg}
 Note that $d$ is the
 multiplicity of the eigenvalue $1$ of the local monodromy action on
 $V$ at $z=0$.  By Levelt's Theorem
 (\cite{levelt},\cite[Theorem~3.5]{bh}) this monodromy has a Jordan
 block of size $d$ (with eigenvalue $1$) and hence cannot be finite
 if $d>1$.  
\end{remark}

\section{Bivariate rational series}
\label{sec:furstenberg}

In this section we discuss Laurent series expansions for rational
functions in two variables.  We prove a lemma which will be of
use in \S \ref{sec:residues} and recall one of the key tools to
determine whether a bivariate series defines a rational function,
namely the observation that a {\it diagonal} of a rational bivariate
series is algebraic.

Let $p(x_1,x_2), q(x_1,x_2) \in \C[x_1,x_2]$ be polynomials in two
variables without common factors and let $f(x_1,x_2) =
p(x_1,x_2)/q(x_1,x_2)$.  We denote by $\NN(q)\subset \R^2$ the Newton
polytope of $q$. Throughout this section we will assume that $\NN(q)$ is two-dimensional.  Let $v_0$ be a vertex of $\NN(q)$, $v_1, v_2$ the
adjacent vertices, indexed counterclockwise and
$\mu_i = v_i - v_0\in \Z^2$, $i=1,2$.  Hence,
\begin{equation}\label{newtoncone}
\NN(q) \subset v_0 + \R_{>0}\cdot \mu_1 + \R_{>0}\cdot \mu_2.
\end{equation}
We can write
$$q(x_1,x_2) \ = x^{v_0} \, (1 - \tilde q(x_1,x_2)),$$
with the support of $\tilde q$ contained in the  cone $\CC\, :=\, \R_{\geq 0}\,\mu_1 + \R_{\geq 0}\,\mu_2$.  Thus we obtain a Laurent expansion of the rational function $f(x)$ as
\begin{equation*}
f(x) \ =\  \sum_{r=0}^\infty \tilde q(x)^r\,,
\end{equation*}
whose support is contained in a cone of the form
$w  + \CC$
for a suitable $w\in \Z^2$.  That is,
$f(x)$ has an expansion
\begin{equation}\label{eq:expansion1}
f(x) = \sum_{m\in \Z^2} a_m \,x^m\,,
\end{equation}
whose support
$\{m\in\Z^2 : a_m\not=0\}$ is contained in $w  + \CC$ for some 
$w\in \Z^2$.

Moreover, the above series converges in a region of the form
\begin{equation} \label{eq:region}
|x^{\mu_1}| < \varepsilon\ ,\quad |x^{\mu_2}| < \varepsilon\,,
\end{equation}
for $\varepsilon$ sufficiently small, as observed in \cite[Proposition 1.5, Chapter 6]{gkzbook}.

\begin{lemma}\label{supportcone}
Given a series (\ref{eq:expansion1}) as above then, for each $i=1,2$, there 
exist infinitely many exponents of the form $m = w_i + r \mu_i$, $w_i\in \Z^2$, $r\in \N$, such that $a_m \not= 0$.  In particular,
the support of the series (\ref{eq:expansion1}) is not contained in any subcone  ${w'} + \CC'$, where $\CC'\, :=\, \R_{\geq 0}\,\mu'_1 + \R_{\geq 0}\,\mu'_2$ is properly contained in $\CC$.
\end{lemma}

\begin{proof} We may assume without loss of generality that $\mu_1 = (s_1,0)$ and that $\mu_2 = (0,s_2)$, $s_1,s_2>0$.  It then suffices to show that for
some $\alpha_0\in \Z$, the series (\ref{eq:expansion1}) contains infinitely many 
terms with non-zero coefficient and  exponent of the form $(\alpha_0,m_2)$, $m_2\in \N$.

We write $p(x_1,x_2) = \sum_{j\geq 0} a_j(x_2) \, x_1^j$, $q(x_1,x_2)
= \sum_{j\geq 0} b_j(x_2) \, x_1^j$ and view them as relatively prime
elements in the ring $\C[x_2,x_2^{-1}][x_1]$.  The Laurent series
expansion (\ref{eq:expansion1}) for the rational function $p(x)/q(x)$
may be written as
$$ \frac {p(x)}{q(x)} \ =\ \sum_{\ell \geq \ell_0} c_\ell(x_2) x_1^{\ell},$$
where $c_\ell(x_2)$ lie in the fraction field of $\C[x_2,x_2^{-1}]$,
that is, the field of rational functions $\C(x_2)$.  Now, it follows
from \cite[Lemma 3.3]{rhf} that since $b_0$ is not a monomial and,
therefore, not a unit in the Laurent polynomial ring
$\C[x_2,x_2^{-1}]$, at least one of the coefficients $c_{\alpha_0}(x_2)$ is
not a Laurent polynomial and, hence there exist infinitely many
non-zero terms with exponents of the form $(\alpha_0,m_2)$.
\end{proof}

Given a bivariate power series
\begin{equation}\label{doubleseries}
f(x_1,x_2) \ :=\ \sum_{n,m\geq 0} a_{m,n} x_1^m x_2^n
\end{equation}
and $\delta = (\delta_1,\delta_2) \in \Z_{>0}^2$, with ${\rm
 gcd}(\delta_1,\delta_2) =1$, we define the $\delta$-diagonal of $f$
as:
\begin{equation}
  \label{diagonalseries}
  f_\delta(t) \ :=\ \sum_{r\geq 0} A_r t^r\, ,\quad A_r:=a_{\delta_1 r,
    \delta_2 r}. 
\end{equation}

The following observation goes back to at least Polya~\cite{Polya}.
We include a proof for the sake of completeness.

\begin{proposition}
\label{diag}
If the series (\ref{doubleseries}) defines a rational function, then
for every $\delta = (\delta_1, \delta_2) \in \Z_{>0}^2$, with
$\gcd(\delta_1, \delta_2) =1$, the $\delta$-diagonal $f_\delta(t)$ is
algebraic.
\end{proposition}

\begin{proof}
 The key observation is that by the one-variable Residue Theorem, we
 can write for $\eta $ and  $t$ small enough
$$
f_\delta(t) = \frac{1}{2\pi i}
\int_{|s|=\eta}\,f\left(s^{\delta_2}t^{\gamma_1},s^{-\delta_1}t^{\gamma_2}\right)\,
\frac {ds} s,
$$
where $\gamma_1,\gamma_2$ are integers such that
$\gamma_1\delta_1+\gamma_2\delta_2=1$. Thus, $f_\delta(t)$, being the
residue of a rational function, is algebraic.
\end{proof}

\begin{remark} We refer the reader to \cite{safonov} for
 generalizations of this result to rational series in more than two
 variables and to Furstenberg~\cite{furstenberg} and
 Deligne~\cite{Deligne} for the situation in characteristic $p>0$ where
 diagonals of rational functions on any number of variables
are algebraic.

\end{remark}

\section{$A$-hyper\-geo\-metric Laurent series}
\label{sec:Laurent}

The Laurent expansions of a rational $A$-hyper\-geo\-metric series are
constrained by the combinatorics of the configuration $A$.  In this
section we sketch the construction of such series.  The reader is
referred to \cite{sst2} for details.

Let $A$ be a regular configuration.  As always, we assume, without loss of
generality, that the points of $A$ are all distinct and that they span $\Z^d$.
We also assume that $A$ is not a pyramid.

We consider the $\C$-vector space:
$$S\  =\  \{\sum_{v\in{\Z^{n}}}\, c_v z^v\ ; \ c_v\in\C\} $$
of formal Laurent series in the variables
$z_1,\ldots,z_{n}$.
The matrix $A$ defines a $\Z^{d}$-valued grading in
$S$ by
\begin{equation}
\deg(z^v) \ :=\  A\cdot v \quad;\quad v\in \Z^{n}\,.
\end{equation}
The Weyl algebra $D_{n}$ acts in the usual
manner on $S$.    We will say that $\Phi\in S$ is
$A$-hyper\-geo\-metric of degree $\beta$ if it is annihilated
by $H_A(\beta)$, i.e.
$$L (\Phi) \ =\  0 \quad \hbox{for all}\quad L\in H_A(\beta).$$
Denote by $\theta =(\theta_1, \dots, \theta_n)$ the vector of
differential operators $\theta_i = z_i {\frac \partial {\partial z_i}}$.
Since for any $v\in\Z^{n}$ we have $(A\cdot\theta) (z^v) =
(A \cdot v) z^v$, it follows that if $\Phi\in S$ is
$A$-hyper\-geo\-metric of degree $\beta$ then it must be
$A$-homogeneous of degree
$\beta$ and, in particular, $\beta\in \Z^d$.

By \cite[Proposition~5]{pst}, if a hyper\-geo\-metric Laurent series
has a non trivial domain of convergence, then its exponents must lie
in a strictly convex cone. We make this more precise.  Let
$$M_\beta \ := \ \{v \in \Z^n : A\cdot v = \beta\}.$$
For any vector $v\in \Z^n$ we define its {\em negative support}
as:
\begin{equation}\label{negativesupport}
{\rm nsupp}(v)\ :=\ \{i\in \{1,\dots,n\} : v_i <0\},
\end{equation}
and given $I \subset \{1,\dots,n\}$, we let
$\Sigma(I,\beta) = \{v\in M_\beta : {\rm nsupp}(v) = I\}$.  We call
$\Sigma(I,\beta)$ a cell in $M_\beta$.

\begin{definition}\label{minimalsupport}
  We say that $\Sigma(I,\beta)$ is a {\em minimal cell} if
  $\Sigma(I,\beta) \not= \emptyset$ and $\Sigma(J,\beta) = \emptyset$
  for $J \subsetneq I$.

\end{definition}

Given a minimal cell $\Sigma(I) =\Sigma(I,\beta)$ we let 
\begin{equation}\label{eq:1.1}
\Phi_{\Sigma(I)}(z) := \sum_{u\in \Sigma(I,\beta)} \ (-1)^{\sum_{i\in I}u_i} \
\frac
{\prod_{i\in I}(-u_i-1)!}{\prod_{j\not\in I}(u_j)!} \
z^{u}\,.
\end{equation}

Given a non zero $w\in \R^n$, $\varepsilon >0$ and $\nu_1, \dots,
\nu_{n-d}$ a $\Z$-basis of the lattice $M$ \eqref{lattice}
satisfying $\langle w, \nu_i \rangle > 0$ for all $i =1, \dots, n-d$
we let $\UU_w \subset \C^n$ be the open set:
\begin{equation}
\label{open}
 |z^{\nu_1}| < \varepsilon\ ,\quad  \dots, \quad  |z^{\nu_{n-d}}| <
\varepsilon.
\end{equation}

The following is essentially a restatement of
Proposition~3.14.13, Theorem 3.4.14, and Corollary 3.4.15 in \cite{sst2}:

\begin{theorem}\label{laurentseries}
  Let $w\in \R^n$ be such that the collection $\Sigma_w$ of minimal
  cells $\Sigma(I,\beta)$ contained in some half-space
$$
\{v\in \R^n: \langle w,v\rangle > \lambda\}, \qquad \lambda\in \R.
$$
is non-empty. Then:

(i) For $\varepsilon$ sufficiently small the open set $\UU_w$ of the
form \eqref{open} is a common domain of convergence of all
$\Phi_{\Sigma(I)}$ \eqref{eq:1.1} with $\Sigma(I)=\Sigma(I,\beta)\in
\Sigma_w$ and

(ii) these $\Phi_{\Sigma(I)}$ are a basis of the vector space of
$A$-hyper\-geo\-metric Laurent series of degree $\beta$
convergent in $\UU_w$.
\end{theorem}

Since an $A$-hyper\-geo\-metric series of degree $\beta$ satisfies $d$
independent homogeneity relations it may be viewed as a function of
$n-d$ variables.  To make this precise we introduce the {\em Gale
 dual} of the configuration $A$.

\begin{definition}\label{def:Gale}
  Let $\nu_1,\dots,\nu_{n-d}\in \Z^n$ be a $\Z$-basis of the lattice
  $M$ \eqref{lattice} and denote by $B$ the $n\times (n-d)$ matrix
  whose columns are the vectors $\nu_j$.  We shall also denote by $B$
  the collection of row vectors of the matrix $B$,
  $\{b_1,\dots,b_n\}\subset \Z^{n-d}$, and call it a {\em Gale dual}
  of $A$. 
\end{definition}

\begin{remark}
\label{Gale-remark}

(i) Our definition of Gale dual depends on the choice of a basis of $M$;
this amounts to an action of $\GL(n-d,\Z)$ on the configuration
$B$.

(ii) $B$ is primitive, i.e., if $\delta \in \Z^{n-d}$ has relatively
prime entries then so does $B\delta$. This follows from the fact that
if $rv\in M$ for $r\in \Z$ and $v\in\Z^n$ then $v\in M$. Equivalently,
the rows of $B$ span $\Z^{n-d}$.

(iii) The regularity condition on $A$ is equivalent to the requirement
  that
\begin{equation}\label{bregularity}
\sum_{j=1}^n b_j \ =\ 0.
\end{equation}

(iv) $A$ is not a pyramid if and only if none of the vectors $b_j$
  vanishes.
\end{remark}

Given $v\in M_\beta$, and the choice of a Gale dual $B$ we may identify
$\,M_\beta \  \cong \  \Z^{n-d}\,$ by
$u\in M_\beta \mapsto m\in \Z^{n-d}$ with
$$u = v + m_1 \nu_1 + \cdots + m_{n-d} \nu_{n-d}.$$
In particular, $u_i < 0$ if and only if $\ell_i(m) <0$, where
\begin{equation}\label{linearforms}
\ell_i(m) :=  \langle b_i,m\rangle + v_i .
\end{equation}
The linear forms in (\ref{linearforms}) define a hyperplane arrangement
oriented by the normals $b_i$ and
each minimal cell $\Sigma(I,\beta)$ corresponds to the closure of a 
certain connected components $\sigma(I)$ in the complement of
this arrangement.

Let $\Phi_{\Sigma(I)}(z)$ as in (\ref{eq:1.1}).
We can also write for $v \in M_\beta$
\begin{equation} \label{eq:1.2a}
\Phi_{\Sigma(I)}(z) =  z^{v} \sum_{ m \in \sigma(I)\cap\Z^2} \ 
\frac
{\prod_{i\in I}(-1)^{\ell_i(m)}( - \ell_i(m)-1)!}{\prod_{j\not\in I}\ell_j(m)!} \
z^{B m}\,.
\end{equation}
Setting
\begin{equation}\label{dehomog}
x_j =  z^{\nu_j},\quad j=1,\dots, n-d,
\end{equation}
 we can now rewrite, the series (\ref{eq:1.1}) in the coordinates $x$ as
$\Phi_{\Sigma(I)}(z) = z^v \varphi_{\sigma(I)}(x)$, where
\begin{equation}\label{eq:1.2}
\varphi_{\sigma(I)}(x) :=   \sum_{ m \in \sigma(I)\cap\Z^2} \
\frac
{\prod_{\ell_i(m)<0}(-1)^{\ell_i(m)}( -
  \ell_i(m)-1)!}{\prod_{\ell_j(m)>0} \ell_j(m)!} \ 
x^{m}\,.
\end{equation}
Moreover, since changing $v\in M_\beta$ only changes
(\ref{eq:1.1}) by a constant, we can assume that in order to write
(\ref{eq:1.2}) we have chosen $v\in \Sigma(I,\beta)$ and this
guarantees that $-v_i-1>0$ for $i\in I$ and $v_j\geq 0$ for
$j\not\in I$.

If $F(z)$ is an $A$-hyper\-geo\-metric function of degree
$\beta$, then $\partial_j (F) = \partial F/\partial z_j $
is $A$-hyper\-geo\-metric of degree $\beta - a_j$. In terms of the
hyperplane arrangement in $\R^{n-d}$ this has the effect changing
the hyperplane $\{ \langle b_j , \cdot \rangle + v_j\}$ to the
hyperplane $\{ \langle b_j , \cdot \rangle + v_j -1 \}$.

The cone of parameters
\begin{equation} \label{EJ}
{\mathcal E}={\mathcal E}_A
\  :=   \  \left\{\sum_{i=1}^{d+2} \lambda_i a_i\,:\,\lambda_i\in
\R,\,\lambda_i<0\right\}
\end{equation}
is called the {\em Euler-Jacobi cone of $A$}.  We note that if $\beta\in
{\mathcal E}$ then $\beta - a_j \in {\mathcal E}$ for all
$j=1,\dots,n$.   

\begin{remark}\label{ejregions}
Given a parameter $\beta$ and 
$\ell_i(x)$  as in \eqref{linearforms} then $\beta\in \EE$ if and only if there
exists a point $\alpha\in \Q^{n-d}$ such that $\ell_i(\alpha)<0$ for all
$i=1,\dots,n$.  This implies in particular that if $b_i,b_j \in B$ are
such that $b_i = -\lambda b_j$, $\lambda>0$, then:
$$\{\ell_i(x) \geq 0\} \cap \{\ell_j(x) \geq 0\} = \emptyset.$$
In particular, all minimal regions $\sigma(I)$ have recession cones of
dimension $n-d$.
\end{remark}

We also recall the following
result \cite[Corollary 4.5.13]{sst2} which we will use in the following sections:

\begin{theorem} \label{th:ejstable}
If $F$ is an $A$-hyper\-geo\-metric function of degree $\beta\in {\mathcal E}$ then, for any $j=1,\dots, n$, $\partial_j (F) =0$ if and only if $F=0$.
\end{theorem}

In particular, all non-zero $A$-hyper\-geo\-metric functions whose degree
lies in the Euler-Jacobi cone are stable.

\begin{example}\label{ex:running}
Let $A\in\Z^{3\times 5}$ be the configuration
\begin{equation}\label{example:configA}
A\  = \  \left(
\begin{array}{ccccc}
1 & 1 & 0 & 0 & 0\\
0 & 0 & 1 & 1 & 1\\
0 & 1 & 0 & 2 & 1
\end{array}
\right)
\end{equation}
$A$ is an essential Cayley configuration of two dimension one configurations:
$A_1=\{0,1\}, A_2=\{0,2,1\}$.
For $\beta\in\C^3$, the ideal  $H_A(\beta)$ is generated by $\pd 1 \pd 4-\pd 2 \pd 5$,
$\pd 3 \pd 4  -  \pd 5^2$,
$\pd 1  \pd 5 - \pd 2  \pd 3$ together with the three Euler operators.  One may verify by direct computation that the function
\begin{equation}\label{example:sol}
F(z) = \frac{z_2}{\,z_1z_2z_5 - z_2^2 z_3 - z_1^2 z_4\,}
\end{equation}
is $A$-hyper\-geo\-metric of degree $\beta = (-1,-1,-1)^t$.
The denominator of $F$ is the discriminant $D_A$ which agrees with the
classical univariate resultant of the polynomials:
\begin{equation}\label{example:pols}
f_1(t) := z_1 + z_2t;\quad f_2(t) := z_3 + z_4t^2 + z_5t.
\end{equation}

A Gale dual of $A$ is given by the matrix:
\begin{equation}\label{example:configB}
B\  = \  \left(
\begin{array}{rr}
-1 & 1 \\
1 & -1 \\
1 & 0 \\
0 & 1 \\
-1 & -1
\end{array}
\right)
\end{equation}
Let $v = (-1,0,0,0,-1)^t$.  Then $A\cdot v = \beta$ and with respect to the inhomogeneous variables:
$$x_1 \,=\,\frac{z_2z_3}{z_1z_5};\quad x_2 \,=\,\frac{z_1z_4}{z_2z_5}$$
we have
$$z^{-v} \, F(z) \,=\, {z_1z_5}\,F(z) \,=\, \frac{1}{1-x_1-x_2}.$$

The hyperplane arrangement associated with $(B,v)$ is defined by the five half-spaces $\ell_i(x) \geq 0$, where  $\ell_1(x) = x_2 - x_1 -1$, $\ell_2(x) = x_1 - x_2 $,
$\ell_3(x) = x_1$, $\ell_4(x) = x_2$, $\ell_5(x) = -x_1 - x_2 -1$.
There are $4$ minimal cells in $M_\beta$, depicted in Figure~\ref{fig:arrangement}.
They are all two-dimensional and correspond to
the negative supports: $I_1 = \{1,5\}$,
$I_2 = \{2,5\}$,
$I_3 = \{2,3\}$, $I_4 = \{1,4\}$.  

\begin{figure}\label{fig:arrangement}
\begin{center}
 {}{}\scalebox{0.60}{\includegraphics{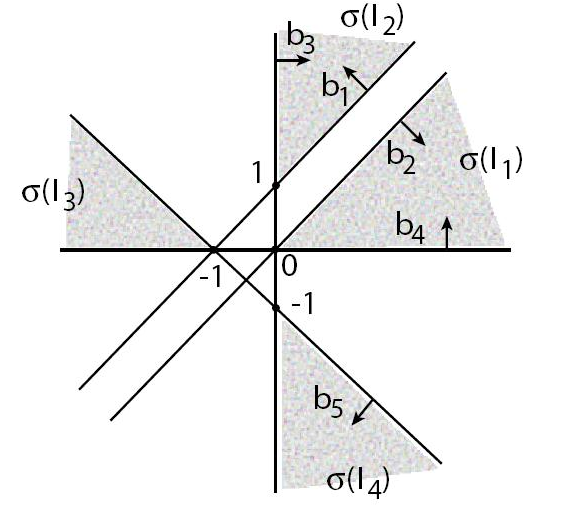}}
 \caption{The hyperplane arrangement corresponding to Example~\ref{ex:running}}
 \end{center}
 \end{figure}

The expansion of $F(z)$ (cf. \eqref{example:sol}) from the vertex
corresponding to $z_1z_2z_5$ in the Newton polytope of the denominator of
$F$ gives:
\begin{eqnarray*}
z_1 z_5 \,F(z)  &=& \sum_{m\in\N^2} \frac {(m_1 + m_2)!}{m_1!\, m_2!}
x_1^{m_1} x_2^{m_2}\\
&=& \sum_{m_1\geq m_2 } \frac {(m_1 + m_2)!}{m_1!\, m_2!}
x_1^{m_1} x_2^{m_2} \ + \sum_{m_2> m_1 } \frac {(m_1 + m_2)!}{m_1!\, m_2!}
x_1^{m_1} x_2^{m_2}\\
&=& \varphi_{\sigma(I_1)} (x) - \varphi_{\sigma(I_2)} (x)\\
\end{eqnarray*}
Similarly, the series $\varphi_{\sigma(I_3)} (x) $ and $ \varphi_{\sigma(I_4)} (x)$ correspond to the expansions from the other two vertices of the Newton
polytope of the denominator of $F$.
\end{example}

\section{Classification of codimension two gkz-rational configurations} \label{sec:classification}

In this section we prove Theorem~\ref{th:conj} classifying all
codimension-two configurations $A$ admitting stable rational
$A$-hyper\-geo\-metric functions.  In particular we obtain a
description of all gkz-rational configurations $A$ of codimension two.

\begin{proof}[Proof of Theorem~\ref{th:conj}]
  As it has already been pointed out, it is shown in \cite{rhf} that
  an essential Cayley configuration is gkz-rational and, therefore,
  admits stable rational $A$-hyper\-geo\-metric functions.  Moreover,
  it follows from \cite{binom} that codimension-two Lawrence
  configurations, while not being gkz-rational, nevertheless admit stable
  rational $A$-hyper\-geo\-metric functions.  Thus we need to consider
  the converse statement; that is, which codimension-two
  configurations admit stable, rational, $A$-hyper\-geo\-metric
  functions.

Suppose then that
$F(z) = P(z)/Q(z)$ is a stable, rational $A$-hyper\-geo\-metric
 function.   $P$ and $Q$ are $A$-homogeneous polynomial and
 consequently, the Newton polytope $\NN(Q)$ of $Q$ lies in a
 translate of $\ker_\R A$.  We choose a vertex $v_Q$ of $\NN(Q)$ and
 a $\Z$-basis $\nu_1,\nu_2$ of $M = \ker_\Z A$ such that
$$\NN(Q) \subset v_Q + \R_{\ge 0}\cdot \nu_1  +  \R_{\ge 0}\cdot \nu_2.$$
Define  $x_i$ as in \eqref{dehomog}, $i=1,2$, and let
$v_P$ be an exponent occurring in $P$.   Then $F$ has $A$-homogeneity $A\cdot(v_P- v_Q)$ and
it has a power series expansion supported in a translate of the cone
$$\CC \ :=\ \R_{\ge 0}\cdot \nu_1  +  \R_{\ge 0}\cdot \nu_2.$$

The basis $\nu_1, \nu_2$ gives rise to a Gale dual $B$ of $A$ as in
Section~\ref{sec:Laurent} and we may choose $v = v_P - v_Q$ to
identify $M_\beta \cong M$.  We can dehomogenize $F$ to get a bivariate rational
function $f(x_1,x_2)=p(x_1,x_2)/q(x_1,x_2)$ which verifies
$$ F(z) \, =\, z^{v_P - v_Q}\ f(x_1, x_2).$$
It follows from Theorem~\ref{laurentseries} that, without loss of
generality,
$$
f(x)=\varphi_{\sigma(I_1)}+c_2\varphi_{\sigma(I_2)}\cdots,
$$
 where
$\sigma(I_1),\sigma(I_2),\ldots$ are minimal cells of the oriented
line arrangement defined by $(B,v)$ contained in the first
quadrant.

Since $F$ is stable, no derivative of $F$ vanishes and, after
appropriate differentiation, we may assume that the degree $\beta$
lies in the Euler-Jacobi cone $\EE$ and, consequently, that there are
no bounded minimal cells of degree $\beta$.  We can also suppose that
$\sigma(I_1)$ is a two-dimensional pointed cone with integral vertex,
which we may assume to be the origin.

For each $\delta = (\delta_1,\delta_2)\in \sigma(I_1)$ with
$\gcd(\delta_1,\delta_2)=1$, the $\delta$-diagonal $f_\delta(t)$ of
$f$ is algebraic.  On the other hand, $f_\delta(t) =
(\varphi_{\sigma(I_1)})_\delta(t)$ and therefore, it follows from
(\ref{eq:1.2}) that:
\begin{equation}\label{eq:diagonal}
 f_\delta(t) =  \pm\, \sum_{ r \geq 0} \
 \frac
 {\prod_{i\in I}( - \langle b_i,\delta\rangle\,r
  -v_i-1)!}{\prod_{j\not\in I}( \langle b_j,\delta\rangle\,r )!} \
 ((-1)^c t)^{r}\,,
\end{equation}
where $c= \langle \sum_{i\in I} b_i ,\delta\rangle$.
Now, according to Theorem~\ref{notalgebraic}, for all $\delta
\in \sigma(I_1)$, the series
\begin{equation}\label{eq:diagonal_central}
  g_\delta(t) =   \sum_{ r \geq 0} \
  \frac
  {\prod_{i\in I}( - \langle b_i,\delta\rangle\,r )!} {\prod_{j\not\in
      I}( \langle b_j,\delta\rangle\,r )!} \   
  t^{r}\,
\end{equation}
is an algebraic function.  We note that, after cancellation, the
coefficients of the series (\ref{eq:diagonal_central}) no longer
involve terms coming from pairs $b_i$, $b_j$ such that $b_i = -b_j$.
We denote by $\tilde B \subset \R^2$ the configuration obtained by
removing all such pairs as well as any zero vector and call it the {\em
  reduced configuration of $B$}.

If $\tilde B = \emptyset$ then $A$ is clearly a Lawrence
configuration.  Next we show that if $A$ admits a stable, rational
hyper\-geo\-metric function then $\tilde B$ cannot be a one-dimensional
vector configuration.  

Indeed, suppose $\tilde B$ is one dimensional, say, $\tilde B
\subseteq \langle \gamma\rangle \subseteq \Z^2$.  Since $f$ is stable,
the Newton polytope $\NN(q)$ is a two-dimensional polytope.  Let $\nu$
be one of its vertices and let $\mu_1$, $\mu_2$ be the adjacent edges.
We may assume without loss of generality that, say, $\mu_1$ is not
orthogonal to $\gamma$.  Consequently, it follows from  Lemma~\ref{supportcone} that the expansion of $f$ from the vertex
$\nu$ contains infinitely many non-zero terms whose exponents lie in a
ray with direction vector $\mu_1$.  The restriction $u$ of $f(x)$ to
such a ray is the specialization of a suitable derivative of $f$ and
hence a one-variable rational function.  On the other hand, $u$ is as
in the hypothesis of Theorem~\ref{notalgebraic} with the $p_i$'s and
$q_i$'s of the form $\langle b, \mu_1\rangle$ for some $b\in \tilde
B$. By (ii) of the Theorem these must cancel out in pairs but then
since $\tilde B$ is one-dimensional, $\tilde B = \emptyset$ which is a
contradiction.

Let 
$$
B_1:= \left(
\begin{array}{rr}
1& 1 \\
-1&0\\
0&-1\\
\end{array}
\right) 
\qquad
B_2:= \left(
\begin{array}{rr}
2& 0 \\
0&2\\
-1&0\\
0&-1\\
-1&-1\\
\end{array}
\right) 
\qquad
B_3:= \left(
\begin{array}{rr}
2& 2 \\
0&1\\
-1&-1\\
-1&0\\
0&-2\\
\end{array}
\right) \,.
$$
By construction of $\tilde B$ there exists infinitely many $\delta\in
\N^2$ with relatively prime entries such that $\tilde B\delta$ has no
zero coordinate and no two distinct coordinates adding up to zero. For
such a $\delta$ there is no cancellation of the factorials in the
coefficients when we take the  $\delta$-diagonal of $f$. Therefore,
since $\tilde B$ has rank two, by the classification of algebraic
hyper\-geo\-metric series in one variable (see
\S\ref{sec:univariate}), there exists two pairs of linearly
independent vectors $\delta,\delta'\in\N^2$ and $\nu,\nu'\in\N^2$ such
that $\tilde B\delta=B_i\nu$ and $\tilde B\delta'=B_i\nu'$ for
some~$i=1,2,3$. In other words, there exists $U\in \GL_2(\Q)$ such
that
$$
\tilde B = B_iU,
$$
for some $i$. In fact, since $B_i$ is primitive, $U\in \Z^{2\times
  2}$.

Now with the notation of the previous paragraph $f$ restricted to the
rays $\mu_1$ and $\mu_2$ is rational. By inspection we see that if
$i=2,3$ there are no two vectors in $\Z^2$ which give restrictions
compatible with Theorem~\ref{notalgebraic} (ii). (i.e., such that the
coordinates of $\tilde B \mu_i$ are of the form $(0,a,-a,b,-b)^t$ for
some $a,b\in\N$ up to permutation.) Therefore $i=1$ and it is now easy
to check that necessarily $A$ is affinely equivalent to an essential
Cayley configuration.  This concludes the proof of
Theorem~\ref{th:conj}.
\end{proof}

\begin{remark}
The simplest series with associated matrix $B_2$ 
$$
u_2(x,y):=\sum_{m,n\geq 0} \frac{(2m)!(2n)!}{m!n!(m+n)!}\,x^my^n
$$
was considered by Catalan. It is an algebraic function, in fact, 
$$
u_2(x,y)=\frac 1{x+y-4xy}\left(\frac x{\sqrt{1-4x}}+\frac
  y{\sqrt{1-4y}}\right)
$$
(see \cite{gessell} for an appearance of this series in
combinatorics).

Similarly, the series
$$
u_3(x,y):=\sum_{m,n\geq 0} \frac{(2m+2n)!n!}{m!(2n)!(m+n)!}\,x^my^n
$$
with associated matrix $B_3$ is algebraic, in fact,
$$
u_3(x,y)=\frac 1{x+4y-xy}\left(\frac x{\sqrt{1-4x}}+\frac
  y{1-y}\right).
$$
(The quickest way to prove these identities is to use a recursion for
the coefficients. For $u_2$ for example 
$$
4A(m+1,n+1)=A(m+1,n)+A(m,n+1),
$$
where $A(m,n):=(2m)!(2n)!/((m+n)!m!n!)$, and
$u_2(t,0)=u_2(0,t)=1/\sqrt{1-4t}$.) 

Note that in addition both $u_2$ and $u_3$ satisfy that all of their
$\delta$-diagonals are algebraic. This is not the typical case for two
variable algebraic functions.

For $B_1$ the natural series is
$$
u_1(x,y):=\sum_{m,n\geq 0} \frac{(m+n)!}{m!n!}\,x^my^n
$$
which is of course rational $u_1(x,y)=1/(1-x-y)$.
\end{remark}

\begin{example}
In \cite[Theorem~4.1]{rhf} it was necessary to show that a
bivariate series of the form:
\begin{equation}\label{series-rhf}
f(x_1,x_2) = \sum_{m\in \N^2}
\frac {(p(m_1+m_2)+k_1)! (q(m_1+m_2)+k_2)!}
{(m_1p)!(m_1q)!(m_2p)!(m_2q)!} \, x_1^m
x_2^n\,,
\end{equation}
where $p,q$ are relatively prime positive integers and $k_1,k_2\in \N$,
does not define a rational function.  While the proof presented in
\cite{rhf} is incomplete, the result is clearly
an immediate consequence
of Theorem~\ref{th:conj}.  Indeed, the  $(1,1)$ univariate diagonal series
$$ \sum_{m\geq 0}
\frac {(2pm+k_1)! (2qm + k_2)!}
{(pm)!^2(qm)!^2} \, z^m
$$
should be algebraic but, then by
 Theorem~\ref{notalgebraic}, so should the central series
$$\ \sum_{m,n\geq 0}
\frac {(2pm)! (2qm)!}
{(mp)!^2(mq)!^2} \, z^m .
$$
 However, this is impossible by \cite[Theorem~1]{frv} since
the height of the series is $2$.
We should point out that the argument in \cite{rhf} was based on a
correct proof for a similar case due to Laura Matusevich.  The gap
appears in adapting that proof to the series \eqref{series-rhf}.
\end{example}

\section{Toric residues and hyper\-geo\-metric functions}\label{sec:residues}

The purpose of this section is to describe all stable
$A$-hyper\-geo\-metric functions in the case of codimension-two
configurations.  By Theorem~\ref{th:conj} we may assume that $A$ is
either a Lawrence configuration or an essential Cayley configuration.
The first case has been studied, for arbitrary codimension, in
\cite{binom}.  In particular, if $A$ is a codimension-two Lawrence
configuration then $A$ is a Cayley configuration of $r+2$ two-point
configurations in $\Z^r$ and it follows from \cite[Theorem~1.1]{binom}
that the dimension of the space of stable $A$-hyper\-geo\-metric
functions is $r+1$ and that they may be represented by appropriate
multidimensional residues.  We refer the reader to \cite{binom} for
details.

Thus, we will restrict ourselves to the case of essential Cayley
configurations.  We begin by recalling the construction of rational
hyper\-geo\-metric functions associated with any essential Cayley
configuration by means of multivariate toric residues (we refer to
\cite{compo, mega96, rhf, binom, cox} for details and proofs) and will
then show in Theorem~\ref{th:dim1} that, in the codimension-two case,
a suitable derivative of any stable rational hyper\-geo\-metric
function must be a toric residue.  In particular, if $\beta\in \EE$,
the dimension of the space of rational $A$-hyper\-geo\-metric
functions is equal to $1$.

Let
\begin{equation*}
A \,\,\, =   \,\, \,
\{e_1\} \! \times A_1 \! \, \, \cup \,\,
\cdots \, \, \cup \,\,
\{e_{r+1}\} \! \times \! A_{r+1}
\,\, \subset \,\, \Z^{r+1} \times \Z^r
\end{equation*}
be an essential Cayley configuration.  For each $A_i \subset \Z^r$
consider the generic Laurent polynomial $f_i$ supported in $A_i$, that
is:
$$
f_i(t) \ =\ \sum_{\alpha \in A_i} \ u_{i\alpha} t^\alpha \ ;\quad
t=(t_1,\dots,t_r).
$$

We set $D_i = \{t\in (\C^*)^r : f_i(t) = 0\}$.  Generically on the
coefficients $u_{i\alpha}$, given any $i=1,\dots,r+1$, the $r$-fold
intersection
$$
V_i \ :=\ D_1 \cap \cdots \cap \widehat{D_i} \cap
\cdots \cap D_{r+1}
$$
is finite and, given any Laurent monomial $t^a$, $a\in \Z^r$, we can
consider the {\em global residue}:

\begin{eqnarray}
\label{globalres}
 R_i(a) & :=  & \sum_{\xi \in V_i}\, {\rm Res}_\xi \left(\frac {t^a/f_i}{f_1\cdots\widehat{f_i}\cdots f_{r+1}} \,\frac{dt_1}{t_1}\wedge \cdots \wedge \frac{dt_r}{t_r}\right)\\
 & =  &
 \label{intres}
 \frac{1}{(2\pi i)^r} \int_\Gamma \frac{t^a}{f_1\cdots f_{r+1}} \,\frac{dt_1}{t_1}\wedge \cdots \wedge \frac{dt_r}{t_r}\,,
\end{eqnarray}
where ${\rm Res}_\xi$ denotes the local Grothendieck residue (see \cite{ghbook,tsikh}) and $\Gamma$ is an appropriate real $r$-cycle on the
torus $(\C^*)^r$.

It is shown in \cite[Theorem~4.12]{compo} that if $a$ lies in the
interior of the Minkowski sum of the convex hulls of $A_1, \dots,
A_{r+1}$ then the expression $(-1)^i R_i(a)$ is independent of $i$.
Its common value is the {\em toric residue} $R(a)$ studied in
\cite{cox,compo}.

It is often useful to consider the expression obtained by replacing in
(\ref{globalres}) the polynomial $f_j$ by $f_j^{c_j}$, where $c_j$ is
a positive integer.  This change defines a function $R_i(c,a), c\in
\Z_{>0}^{r+1}$, and if $a$ lies in the interior of the Minkowski sum
of the convex hulls of $c_1\,A_1, \dots, c_{r+1}\,A_{r+1}$ then the
expression $(-1)^i R_i(c,a)$ is independent of $i$.

The toric residue $R(c,a)$ is a rational function on the coefficients
$u_{i\alpha}$ and is $A$-hyper\-geo\-metric of degree $\beta = (-c,-a)
\in \Z^{r+1}\times Z^r$.  This may be seen, for example, by
differentiating under the integral sign in the expression
(\ref{intres}). We refer to \cite[Theorem~7]{mega96} for details. 
 
It follows from the arguments in \cite[\S 5]{rhf}, that the function
$R(c,a)$ does not vanish and, since for a given 
$c\in \Z^{r+1}_{>0}$,
a point $a\in \Z^r$ is in the interior of the Minkowski sum of
$c_1\,A_1, \dots, c_{r+1}\,A_{r+1}$ if and only if
$(-c,-a)$ lies in the Euler-Jacobi cone $\EE$  (\ref{EJ}), it
follows that $R(c,a)$ is a stable rational $A$-hyper\-geo\-metric function in this case.

We note that
\begin{equation}\label{derivres}
  \frac {\partial R(c,a) }{\partial u_{i\alpha}} = -c_i\,R(c+e_i,a+\alpha),
\end{equation}
where $e_i$ denotes the $i$-th vector in the standard basis of $\Z^{r+1}$.

The family of Laurent polynomials $f_1,\dots,f_{r+1}$ associated with
a codimension-two Cayley essential configuration must consist of $r$
binomials and one trinomial.  Thus, after relabeling the coefficients
and an affine transformation of the exponents we may assume that
\begin{eqnarray} \label{cayley7}
f_i \quad =   & z_{2i-1} \, + \, z_{2i}\cdot t_1^{\gamma_i}\ ;\qquad i=1,\dots,r
\nonumber\\
f_{r+1} \quad =   & z_{2r+1} \, + \,
z_{2r+2} \cdot  t^{\alpha_1} \,+ \,
z_{2r+3} \cdot  t^{\alpha_2} ,\nonumber
\end{eqnarray}
with $\gamma_1, \dots, \gamma_r,\alpha_1,\alpha_2 \in \Z^r\backslash\{0\}$.  

\begin{theorem}\label{th:dim1}
  Let $A \subset \Z^{2r+1}$ be a codimension-two Cayley essential
  configuration and suppose $\beta=(-c,-a) \in \EE \cap \Z^2$.  Then, any
  rational $A$-hyper\-geo\-metric function of degree $\beta$ is a
  multiple of $R(c,a)$.
\end{theorem}

\begin{proof}  
  The Gale dual $B$ of a codimension-two Cayley essential
  configuration is a collection of $2r+3$ vectors
  $b_1,\dots,b_{2r+3}\in \Z^2$ which, after renumbering, may be
  assumed to be of the form:
$$
b_1+b_2\ = \ \dots \ = \ b_{2r-1}+b_{2r}\ =\ b_{2r+1}+b_{2r+2}
+b_{2r+3} \ =\ 0,
$$
where the vectors $b_{2r+1}, b_{2r+2}, b_{2r+3}$ are not
collinear. 


As in \eqref{linearforms} we denote
by $\ell_i(x) = \langle b_i,x\rangle + v_i$ the linear functionals in $\R^2$ defined by $B$ and a choice of $v\in \Z^{2r+3}$ such that
$A\cdot v = \beta$.  For $1\leq i < j \leq 3$, we set
\begin{equation}\label{eq:bigregions}
\Lambda_{ij} \,:=\,\{x\in \R^2 : \ell_{2r+i}(x) \geq 0,\  \ell_{2r+j}(x) \geq 0\}.
\end{equation}

Let $F(z) = P(z)/Q(z)$ be any non-zero (and hence, stable) $A$-hyper\-geo\-metric function of degree
$\beta$ and write
$z^{-v}\cdot F(z) = f(x)=p(x)/q(x)$, where $x = (x_1,x_2)$ is as in \eqref{dehomog}. 
 
We claim that the Newton polytope $\NN(q)$ of the polynomial
$q(x)$ is a triangle whose inward pointing normals are the vectors $b_{2r+1}$, 
$b_{2r+2}$, $b_{2r+3}$ (this is indeed the case for the residue
$R(c,a)$ since its denominator is a power of the discriminant $D_A$, whose Newton polytope
is such a triangle  by \cite{elim}).  Let $\nu_0$ be a vertex of $\NN(q)$ and
$\nu_1,\nu_2$ the adjacent vertices.  Set $\mu_i = \nu_i - \nu_0$, $i=1,2$.   
The Laurent expansion of  $f(x)$ from the vertex $\nu_0$ is supported in a cone of the form 
$$\CC = c_0 +\R_{\geq 0}\cdot \mu_1 + \R_{\geq 0}\cdot \mu_2\ ;\quad c_0\in\Z^2\, .
$$
On the other hand, since $f(x)$ is the dehomogenization of an $A$-hyper\-geo\-metric function, it follows from Theorem~\ref{laurentseries} that $f$ may be
written as 
\begin{equation}\label{eq:expansion}
f \,=\, \sum_\sigma \alpha_\sigma \,\varphi_\sigma\,,
\end{equation}
where $\sigma$ runs over all minimal regions of the hyperplane arrangement
defined by the linear functionals $\ell_i(x)$ which are contained in the cone $\CC$, and $\alpha_\sigma \in \C$.  

Let $\sigma$ be any region appearing in \eqref{eq:expansion} with a non-zero coefficient.  
Since $\sigma$ is minimal region, all linear forms $\ell_i$ have a constant sign in
its interior. As noted in Remark~\ref{ejregions}, $\sigma$ must have a two-dimensional recession cone.  
Let $\delta$ be a rational direction in the interior of 
$\sigma$ and consider the $\delta$-diagonal
of $\varphi_\sigma$.  By Proposition~\ref{diag}, the function 
$(\varphi_\sigma)_\delta(t)$ must be algebraic. 
As the series $\varphi_\sigma$ has the form~\eqref{eq:sigma} below, it follows from the
discussion in Section~\ref{sec:univariate} that this can only happen if two of the
linear forms $\ell_{2r+j}, j=1,2,3$, are positive (and the third one negative) over the interior
of $\sigma$. This proves that $\sigma$ is contained in one of the regions $\Lambda_{ij}$. 

Now, it follows from Lemma~\ref{supportcone} 
that there must be minimal regions $\sigma_1, \sigma_2\subset \CC$, not necessarily distinct, appearing with non-zero coefficients in the expansion \eqref{eq:expansion} such that $\sigma_i$ contains  all points of the form
$c_i + k \mu_i$, $i=1,2$, for suitable $c_1,c_2\in \Z^2$ and $k\in\Z_{>0}$ 
sufficiently large.

Consider the series $\varphi_{\sigma_1}(x)$ associated to the minimal region 
$\sigma_1$.  It follows from \eqref{ejregions} that 
\begin{equation} \label{eq:sigma}
\varphi_{\sigma_1}(x) = \sum_{m\in \sigma_1} h(m) \frac{\prod_{\ell_{2r+j}(m)<0}(-\ell_{2r+j}(m)-1)!}{\prod_{\ell_{2r+j}(m)>0}\ell_{2r+j}(m)!} x^m,
\end{equation}
where $h(m)$ is a polynomial.
But, since $f$ is a rational function we deduce that the univariate function
$$ \sum_{k>>0} h(c_1 + k \mu_1) \frac{\prod_{\langle b_{2r+j},\mu_1\rangle <0}(-\langle b_{2r+j}, \mu_1\rangle\,k -\langle b_{2r+j}, c_1\rangle -1)!}{\prod_{{\langle b_{2r+j},\mu_1\rangle >0}}(\langle b_{2r+j}, \mu_1\rangle\,k +\langle b_{2r+j}, c_1\rangle)!} t^k $$
must be a rational function.  But by item ii) in Theorem~\ref{notalgebraic} this is only possible if 
$$\langle b_{2r+j},\mu_1\rangle =0$$
for some $j=1,2,3$. As the ray with direction $\mu_1$ is in the boundary of
the minimal region $\sigma_1$, this implies that $\sigma_1$ cannot be contained
in $\Lambda_{ik}$, where $i,k\not=j$. Consequently,
$b_{2r+j}$ is an inward pointing normal to $\NN(q)$  and our claim is proved.

Given now any rational hyper\-geo\-metric function $F = P/Q$ of degree $\beta$
we may now consider the Laurent expansion of its dehomogenization $f=p/q$ from the vertex of $\NN(q)$ defined by the
edges with inward-pointing normals $b_{2r+1}$ and $b_{2r+2}$.  When that expansion is written as in \eqref{eq:expansion} there must be, by Lemma~\ref{supportcone}, a minimal region $\sigma$ whose recession cone has a boundary
line orthogonal to $b_{2r+1}$ and the corresponding coefficient $a_\sigma$ must be non-zero.
Hence, the map $F \mapsto a_\sigma$ is $1:1$ and the space of rational $A$-hyper\-geo\-metric functions of degree $\beta$ has dimension at most one.  As we have already recalled,  it follows from  \cite[\S 5]{rhf} that
the toric residue $R(c,a)$ is a non zero rational $A$-hyper\-geo\-metric
function of degree $\beta$, which thus spans the vector space of all rational
$A$-hyper\-geo\-metric functions of this degree.
\end{proof}

\begin{example}\label{ex:running3}
We continue with Example~\ref{ex:running}.
Let $F(z)$ be as in \eqref{example:sol} and $f_1, f_2 \in \C[t]$  as in \eqref{example:pols}.  Then we have
$$F(z)\, =\, -R(1) \,=\, -{\rm Res}_{-z_1/z_2} \left(\frac {dt/f_2(t)}{f_1(t)}\right).$$
We showed in Example~\ref{ex:running} that
in inhomogeneous coordinates
$$ z_1 z_5 \,F(z) \,=\, \varphi_{\sigma(I_1)} (x) - \varphi_{\sigma(I_2)} (x)$$
for the minimal regions $\sigma(I_1), \sigma(I_2)$ contained in the first quadrant.  According to Theorem~\ref{th:dim1} neither $\varphi_{\sigma(I_1)} (x)$ nor
$\varphi_{\sigma(I_2)} (x)$ can be rational functions.  Indeed, one can
check by direct computation that, up to sign,  $\varphi_{\sigma(I_1)} (x)$ and 
$\varphi_{\sigma(I_2)} (x)$ agree with the pointwise residues:
$${\rm Res}_{\xi_\pm} \left(\frac {dt/f_1(t)}{f_2(t)}\right),$$
where $\xi_\pm$ are the roots of $f_2(t)$: 
$$ \xi_\pm\,:=\, \frac{-z_5\pm \sqrt{z_5^2-4z_3z_4}}{2z_4}$$
and, in the inhomogeneous coordinates $x_1$, $x_2$, we have:

\begin{eqnarray*}
\varphi_{\sigma(I_1)} (x)  &=&  \sum_{m_1\geq m_2\geq 0 } \frac {(m_1 + m_2)!}{m_1!\, m_2!}\
x_1^{m_1} x_2^{m_2} \\
&=& \frac{1}{2(1-x_1-x_2)} \left(1+\frac{1-2x_2}{\sqrt{1-4x_1x_2}}\right)\,.
\end{eqnarray*}

\end{example}

\section{Classical bivariate rational hyper\-geo\-metric series}\label{sec:rat-bivariate}

In this section we will apply the previous results  to study the rationality of power series
in two variables which generalize the univariate series discussed in
Section~\ref{sec:univariate}, that is, series
 whose coefficients are ratios of products of factorials
 of linear forms  defined over $\Z$.  

Our starting data will be a support cone $\CC$ which will be assumed to
be a two-dimensional rational,  convex polyhedral cone in $\R^2$ and linear functionals
\begin{equation}\label{functionals}
\ell_i(x) \,:=\, \langle b_i,x\rangle + k_i\,,\quad i=1,\dots,n,
\end{equation}
where $b_i \in \Z^2\backslash \{0\}$, $k_i\in \Z$. We will denote by
$\mu_1,\mu_2$ the primitive integral vectors defining the edges of $\CC$ and by $\nu_1,\nu_2 \in \Z^2$ the corresponding primitive inward normals.

\begin{definition} \label{def:Horn}
Given $\CC$ and $\ell_i$, $i=1, \dots,n$ as above, the bivariate series:
\begin{equation}\label{eq:horn}
 \sum_{ m\in\CC\cap\Z^2}
 \frac{\prod_{\ell_i(m) < 0} \ (-1)^{\ell_i(m)}\, (-\ell_i(m)-1)!}{\prod_{\ell_j(m) > 0} \ \ell_j(m)!}
\,
x_1^{m_1} x_2^{m_2}.
\end{equation}
will be called a {\em Horn series}. 
\end{definition}

\begin{remark}\label{rem:Horn}
  Let $\phi(x_1, x_2) = \sum_{ m\in\CC\cap\Z^2} c_m x^m$ be a Horn
  series as in \eqref{eq:horn}.  Then, the coefficients $c_m$ satisfy
  a Horn recurrence; that is, for $j=1,2$, and any $m \in \CC\cap\Z^2$
  such that $m +e_j$ also lies in $\CC$, the ratios:
  \[ R_j(m) \, := \,\frac{c_{m+e_j}}{c_m}\,=\, \frac{ \prod_{b_{ij} <
      0} \prod_{l=0}^{-b_{ij} + 1} \ell_i(m) - l} {\prod_{b_{ij} > 0}
    \prod_{l=1}^{b_{ij}} \ell_i(m) + l}.\] are rational functions of
  $m$ (recall that $e_j$ denote  the standard basis vectors). 
\end{remark}

We are interested in studying when a Horn series defines a rational
function $\phi(x_1, x_2)$.  We will assume that
\begin{equation}\label{eq:sum}
\sum_{i=1}^n b_i \ =\  0 \,,
\end{equation}
and note that \eqref{eq:sum} implies that \eqref{eq:horn} converges
for $|x^{\mu_1}|<\varepsilon$, $|x^{\mu_2}|<\varepsilon$ for any small
$\varepsilon > 0$. 

\begin{remark}\label{horn=hypergeom}
Every Horn series \eqref{def:Horn} is the dehomogenization of an $A$-hyper\-geo\-metric function for some regular configuration $A$.  More precisely,  there exists a codimension-two configuration $A\subset Z^{s-2}$, a vector 
$v\in\Z^s$,  a $\Z$-basis $\nu_1,\nu_2$ of $\ker_\Z(A)$, and an $A$-hyper\-geo\-metric function $F(z)$  of degree $A\cdot v$ such that:
$$ F(z_1,\dots,z_s) \,=\,z^v\ \phi(z^{\nu_1},z^{\nu_2})\,.$$

This may be seen as follows:  the linear forms $\ell_1,\dots,\ell_n$ define an oriented hyperplane arrangement associated with the vector configuration
$$B\, =\, \{b_1,\dots,b_n\}$$
and the vector $(k_1,\dots,k_n)$.  We can enlarge $B$ to a new configuration
$\hat B$ by adding to $B$ pairs of vectors
$\{c,-c\}$ where $c$ ranges over all $b_i\in B$, $\nu_1,\nu_2$, and the
standard basis vectors $e_1$, $e_2$.  $\hat B$ is the Gale dual of a 
configuration $A$ and, for a suitable choice of parameter $\hat v\in \Z^s$,
$s=|\hat B|$, every region in the hyperplane arrangement defined by $(\hat B, \hat v)$ is minimal.  and the series \eqref{eq:horn} is the dehomogenization of an $A$-hyper\-geo\-metric
series of degree $A\cdot \hat v$.
\end{remark}

The following theorem characterizes rational bivariate Horn series:

\begin{theorem}\label{th:conjclassic}
Let $\ell_i(x) = \langle b_i,x\rangle + k_i$, $i=1,\dots,n$, be linear forms
on $\R^2$ defined over $\Z$ and $\CC$ a two-dimensional rational, convex, polyhedral cone
in $\R^2$. Let
\[
\phi(x_1,x_2) = \sum_{ m\in\CC\cap \Z^2}
 \frac{\prod_{\ell_i(m) < 0} \ (-1)^{\ell_i(m)}\, (-\ell_i(m)-1)!}{\prod_{\ell_j(m) > 0} \ \ell_j(m)!}
\,
x_1^{m_1} x_2^{m_2}.
\]
be a Horn series satisfying \eqref{eq:sum}.
Set $B=\{b_1,\dots,b_n\}\subset \Z^2$.
If
$\phi(x_1,x_2)$ is a rational function then either
\begin{itemize}
\item[(i)]  $n = 2r$ is even and, after reordering we may assume:
\begin{equation}\label{eq:case1}
b_1 + b_{r+1} = \cdots = b_{r} + b_{2r } = 0,\ or
\end{equation}
\item[(ii)] $B$ consists of $n = 2r+3$  vectors and,
after reordering, we may assume that $b_1,\dots,b_{2r}$ satisfy \eqref{eq:case1}
and $b_{2r+1} = s_1 \nu_1$, $b_{2r+2} = s_2 \nu_2$, $b_{2r+3} = -b_{2r+1}-b_{2r+2}$, where $\nu_1, \nu_2$ are the primitive, integral, inward-pointing normals of $\CC$ and $s_1,s_2$ are positive integers.
\end{itemize}
\end{theorem}

\begin{proof} As noted in Remark~\ref{horn=hypergeom}, the series 
$\phi(x_1,x_2)$
may be viewed as the dehomogenization of an $A$-hyper\-geo\-metric function, for a suitable regular configuration $A$ whose Gale dual $\hat B$ is obtained
from $B$ by adding pairs of vectors $\{c,-c\}$, $c\in \Z^2$. 
Since $A$ admits a stable rational hyper\-geo\-metric function, it follows from
Theorem~\ref{th:conj} that $A$ is either a Lawrence configuration or a Cayley essential configuration.  It is now clear that in the first case, $B$ satisfies \eqref{eq:case1}, while in the second, $n=2r+3$ and we may assume that
$b_1,\dots,b_{2r}$ also satisfy \eqref{eq:case1}, while
$$b_{2r+1} +b_{2r+2} +b_{2r+3} =0$$

Moreover, if $A$ is a Cayley essential configuration then it is shown in the proof of Theorem~\ref{th:dim1} that
$$\phi(x) \,=\, \sum a_\sigma \varphi_\sigma(x),$$
where $\varphi_\sigma(x)$ are  canonical series, as in \eqref{eq:1.2}, associated with the minimal regions of the hyperplane arrangement of $\hat B$,
and the sum runs over all minimal regions contained in one of the sectors defined
by  the half-spaces $\ell_{2r+1}\geq 0$, $\ell_{2r+2}\geq 0$, $\ell_{2r+3}\geq 0$.  
But then, since the expansion \eqref{eq:horn} is not supported in any proper subcone of $\CC$, it follows that $\CC$ must agree with one of those sectors. Hence, after reordering if necessary, we have that
$$b_{2r+1} = s_1 \nu_1,\quad b_{2r+2} = s_2 \nu_2.$$
\end{proof}

\begin{example}\label{ex:gessell}
As an illustration of the type of series  Theorem~\ref{th:conjclassic} refers to
consider the following expansion  from \cite{gessellx}[Example 9.2]
$$
\varphi(x) = \frac{1-x_1x_2}{1-x_1x_2^2-3x_1x_2-x_1^2x_2}=
\sum_{m\in \CC\cap \Z^2} \binom{m_1+m_2}{2m_1-m_2}\,x_1^{m_1}x_2^{m_2},
$$
where $\CC:=\{2m_1-m_2\geq 0, 2m_2-m_1\geq 0\}$.

The series
\begin{equation} \label{eq:varphi}
\phi(x) = \varphi( - x) = \sum_{m\in \CC\cap \Z^2} (-1)^{m_1 + m_2} \binom{m_1+m_2}{2m_1-m_2}\,x_1^{m_1}x_2^{m_2}
\end{equation}
is a Horn series. It follows from Theorem~\ref{th:dim1} that $\phi(x)$  may be represented
as a residue.  Indeed, following the notation of Theorem~\ref{th:conjclassic},
the configuration $B$ is defined by the vectors $b_1=(-1,-1)$, $b_2=(-1,2)$, $b_3=(2,-1)$.  We enlarge it to a configuration $\hat B$ by adding the vectors $b_4=(1,0)$, and 
$b_5 = (-1,0)$.  Now, $\hat B$ is the Gale dual of the Cayley essential 
configuration
$$A \ =\ \left(
\begin{array}{rrrrrrr}
1& 1 &1 &0&0\\
0&0&0&1&1\\
0&1&2&0&3\\
\end{array}
\right) 
$$
and $\phi(x)$ is the dehomogenization of an $A$-hyper\-geo\-metric toric residue associated to
$f_1=z_1 + z_2 t + z_3 t^2, f_2 = z_4+z_5 t^3$.
Explicitely, in inhomogeneous coordinates we have:
\[\phi(x) \, = \, \sum_\eta {\rm Res}_\eta \left( \frac {x_2 t/(x_2+x_2 t- t^2)}{x_2+x_1t^3} \,dt \right),\]
where $\eta$ runs over the three cubic roots of $-x_2/x_1$; that is,
$\phi$ is the global residue with respect to the family of polynomials
$x_2 + x_1 t^3$ of the rational function of $t$ (depending
parametrically on $x$) defined by $t/(1+ t- x_2^{-1} t^2)$.
\end{example}

In the remainder of this section we consider the special case where
$\CC$ is the first quadrant.
The following  series will play a central role in our discussion.

\begin{proposition}\label{prop:biv-rational}
The series
\begin{equation}\label{eq:biv-cayley}
f_{(s_1,s_2)}(x)\,  := \, \sum_{m\in\N^2} \frac{(s_1 m_1 + s_2 m_2)!}{(s_1
 m_1)!(s_2 m_2)!}\  x_1^{m_1} x_2^{m_2}\,.
 \end{equation}
defines a rational function for all $(s_1,s_2)\in\N^2$.
\end{proposition}

\begin{proof}
The assertion is evident if either $s_1=0$ or $s_2=0$ since in this case
\eqref{eq:biv-cayley} becomes:
\begin{equation}\label{eq:biv-lawrence}
f_0(x_1,x_2)\  =\  \sum_{m\in\N^2} x_1^{m_1} x_2^{m_2}\ =\ \frac{1}{(1-x_1)(1-x_2)}\,,
\end{equation}
as well as in the case when $s_1=s_2=1$ since
\[f_{(1,1)}(x) \, = \, \, \sum_{m\in\N^2} \frac{(m_1 + m_2)!}{
 m_1!\, m_2!}\  x_1^{m_1} x_2^{m_2}\,=\,\frac{1}{1-x_1-x_2}.
 \]
More in general, given any $s_1, s_2 >0$, 
consider the Cayley essential configuration:
$$A \ =\ \left(
\begin{array}{rrrrrrr}
1& 1 &1 &0&0&0&0\\
0&0&0&1&1&0&0\\
0&0&0&0&0&1&1\\
1&0&0&0&s_1&0&0\\
0&1&0&0&0&0&s_2\\
\end{array}
\right) 
$$
and $\beta=(-1,-1,-1,-s_1,-s_2)^t = A\cdot (0,0,-1,0,-1,0,-1)^t$. 
Consider the hyperplane arrangement associated with the vector $(0,0, -1, 0, -1, 0, -1)^t$ and the Gale dual
$B$ of $A$ with rows $b_1 =(s_1,0), b_2 = (0, s_2), b_3= (-s_1, -s_2)$, $b_4 = (1,0) = - b_5$,
$b_6= (0,1) = - b_7$. The first quadrant is a minimal region and the 
corresponding Laurent $A$-hypergeometric series is:
$$F(z) \ =\ \frac {1}{z_3z_5z_7}\sum_{m\in \N^2} \frac{(s_1m_1+s_2m_2)!}{(s_1m_1)!(s_2m_2)!} \left(\frac{(-z_1)^{s_1}z_4}{z_3^{s_1}z_5}  \right)^{m_1}\left(\frac{(-z_2)^{s_2}z_6}{z_3^{s_2}z_7}  \right)^{m_2}.$$
Thus, the rationality $F$ implies that of $f_{(s_1,s_2)}$.  But,
since the first quadrant is the only minimal
region contained
in the open half space $\{s_1 m_1 + s_2 m_2 \ge 0\}$ and $A$ is a Cayley essential configuration, the series $F(z)$ must
agree with a Laurent expansion of the toric residue
\[   {\rm Res}\left(
\frac{t_1^{s_1}t_2^{s_2}/(z_1t_1+z_2t_2+z_3)}{(z_4+z_5t_1^{s_1})(z_6+z_7t_2^{s_2})} \, \frac{dt_1}{t_1}\wedge\frac{dt_2}{t_2}\right),\]
which is a rational function.  In fact, we can write explicitly:
\begin{equation}\label{sres}
f_{(s_1,s_2)}(x)\,=\,\sum_{\xi_1^{s_1} = - x_1, \xi_2^{s_2} = - x_2} {\rm Res}_\xi \left(
\frac{t_1^{s_1}t_2^{s_2}/(t_1+t_2+1)}{(x_1+t_1^{s_1})(x_2+t_2^{s_2})} \, \frac{dt_1}{t_1}\wedge\frac{dt_2}{t_2}\right).
\end{equation}
\end{proof}

An alternative proof of Proposition~\ref{prop:biv-rational} follows from the the fact that $f_{(1,1)}$ is rational together with the following two lemmas, which are of independent interest. 

\begin{lemma}\label{lemma1}
Suppose
$$
\sum_{m\in\N^2}a(m_1,m_2)\, x_1^{r_1m_1}x_2^{r_2m_2},
$$
for some  fixed positive integers $r_1,r_2$,
is the Taylor expansion of a rational function $f(x_1,x_2)$. Then the same is true of
$$
\sum_{m\in\N^2}a(m_1,m_2)\, x_1^{m_1}x_2^{m_2}.
$$
\end{lemma}
\begin{proof}
Write $f=A/B$, where  $A,B$ are relatively prime polynomials in
$\C[x_1,x_2]$. For any $\zeta_1, \zeta_2\in \C^2$ such that
$\zeta_1^{r_1}=\zeta_2^{r_2}=1$ we have
$$
A(\zeta_1x_1,\zeta_2x_2)\,B(x_1,x_2)=
A(x_1,x_2)\,B(\zeta_1x_1,\zeta_2x_2).
$$
Hence
$$
A(\zeta_1x_1,\zeta_2x_2)= c\,A(x_1,x_2), \qquad
B(\zeta_1x_1,\zeta_2x_2)= c\,B(x_1,x_2)
$$
for some non-zero constant $c\in \C$. Since $B(0,0)$ is nonzero (as we
assume $f$ is holomorphic at the origin) evaluating at $(0,0)$ shows
$c=1$ and the result follows.
\end{proof}

\begin{lemma} \label{lemma:r1r2}
Suppose
$$
\sum_{m\in\N^2}a(m_1,m_2)\, x_1^{m_1}x_2^{m_2}
$$
is the Taylor expansion of a rational function $f(x_1,x_2)$. Then
the same is true of
$$
\sum_{m\in\N^2}a(r_1m_1,r_2m_2)\, x_1^{m_1}x_2^{m_2},
$$
for any fixed $(r_1,r_2)\in \Z_{>0}^2$.
\end{lemma}
\begin{proof}
Note that
\begin{equation} \label{eq:r1r2bis}
\sum_{m\in\N^2}a(r_1m_1,r_2m_2)\, x_1^{r_1m_1}x_2^{r_2m_2}=
\frac1{r_1r_2}\sum_{\zeta_1^{r_1}=\zeta_2^{r_2}=1}f(\zeta_1 x_1,\zeta_2 x_2),
\end{equation}
for sufficiently small $|x_1|$ and $|x_2|$. The right hand side is clearly
a rational function. Hence our claim follows from Lemma~\ref{lemma1}.
\end{proof}

Note that the local sum of residues ~\eqref{sres}
has the same form as the sum in \eqref{eq:r1r2bis} in
the proof of Lemma~\ref{lemma:r1r2}.

\begin{example}
Consider the case $s=(2,2)$. By definition
$$
f_{(2,2)}(x_1,x_2)=\sum_{m\in\N^2}
\frac{(2m_1+2m_2)!}{(2m_1)!(2m_2)!}\,x_1^{m_1}x_2^{m_2}.
$$
Equality \eqref{eq:r1r2bis} reads
$$
f_{(2,2)}(x_1^2,x_2^2)=\tfrac 14
(f(x_1,x_2)+f(-x_1,x_2)+f(x_1,-x_2)+f(-x_1,-x_2)),
$$
where
$
f(x):=f_{(1,1)}(x)$.
Then,
$$
f_{(2,2)}(x_1^2, x_2^2) \, = \, \frac{1-x_1^2-x_2^2}{1-2x_1^2-2x_2^2-2x_1^2x_2^2+x_1^4+x_2^4},
$$
and hence
$$
f_{(2,2)}(x_1,x_2)=\frac {1-x_1-x_2}{1-2x_1-2x_2-2x_1x_2+x_1^2+x_2^2}.
$$
\end{example}

Our last result shows that, up to the action of differential operators of the form
\ref{eq:P} below, all rational Horn series with support on all the integer points of the
first quadrant are given by the functions $f_{(s_1, s_2)}$.

\begin{theorem}\label{th:conjclassic2}
Let $\ell_i(x) = \langle b_i,x\rangle + k_i$, $i=1,\dots,n$, be linear forms
on $\R^2$ defined over $\Z$ and suppose that the Horn series
\[
\phi(x_1,x_2) = \sum_{ m\in\N^2}
 \frac{\prod_{\ell_i(m) < 0} \ (-1)^{\ell_i(m)}\, (-\ell_i(m)-1)!}{\prod_{\ell_j(m) > 0} \ \ell_j(m)!}
\,
x_1^{m_1} x_2^{m_2}.
\]
 satisfies \eqref{eq:sum} and defines 
 a rational function.

Then,  there exist  differential operators $P_1(\theta), P_2(\theta)$ of the form
\begin{equation}
\label{eq:P}
 \prod_j \langle b_j,\theta\rangle + c_j,  \quad \theta = (\theta_1, \theta_2), \,
 \theta_i = x_i\, \partial/\partial x_i\,; \  c_j\in \Z,
\end{equation}
such that $$P_1(\theta) \cdot \phi(x)\, =\, P_2(\theta)\cdot f_{(s_1, s_2)}(\pm x_1, {\pm} x_2),$$
 where $s_1=s_2=0$ in case $B=\{b_1,\dots,b_n\}$ is a Lawrence configuration  and $s_1, s_2 > 0$ if $B$ is Cayley essential.
\end{theorem}

\begin{proof}
It follows from Theorem~\ref{th:conjclassic} that $B$ must be either a Lawrence or a Cayley essential configuration.  In the latter case, we 
have moreover that $n=2r+3$, $b_1,\dots,b_{2r}$ are as in 
\eqref{eq:case1} while
$b_{2r+1} = (s_1,0)$, $b_{2r+2} = (0,s_2)$, $b_{2r+3} = (-s_1,-s_2)$ for
$s_1,s_2$ positive integers.  Therefore,  we can find
a differential operator $P_1(\theta)$ as in \eqref{eq:P} such that
$$P_1(\theta)\cdot \phi(x) = \pm \sum_{m\in \N^2} \left(\prod_{i=1}^r\prod_{j=c_i}^{d_i} (\langle b_i ,m\rangle + j) \right) \frac {(s_1m_1 + s_2m_2 +k)!}{(s_1m_1)!(s_2m_2)!} (\pm x_1)^{m_1} (\pm x_2)^{m_2},$$
for suitable integers $c_i,d_i$.  Thus taking
$$P_2(\theta) \,=\, \pm\left(\prod_{i=1}^r\prod_{j=c_i}^{d_i} (\langle b_i ,\theta\rangle + j) \right)\, \prod_{j=1}^k (s_1\theta_1 + s_2\theta_2 +j),$$
we get
$$P_1(\theta)(\phi(x)) = P_2(\theta)(f_{(s_1,s_2)}(\pm x_1, \pm x_2)).$$
The argument in the Lawrence case is completely analogous.
\end{proof}

\begin{example}
We return to the rational function in  Example~\ref{ex:gessell}:
$$\phi(x) \, = \, \frac{1-x_1x_2}{1+x_1x_2^2-3x_1x_2+x_1^2x_2} ,$$
and its Laurent expansion \eqref{eq:varphi}.  Let $m_1'=2m_1-m_2,
m_2'=2m_2-m_1$ (so that $m_1 = \frac{2m_1'+m_2'}3, m_2 =
\frac{m_1'+2m_2'}3$) then
\[\phi(x) \, = \, \sum_{{{(m_1',m_2')\in L\cap\N^2}}} \frac{ (m_1'+
  m_2')!}{m_1'!m_2'!}\ u_1^{m_1'} u_2^{m_2'},
\]
where $L$ denotes the sublattice $L =\Z (1,2) + \Z (2,1) =
\{(m_1',m_2')\in \Z^2 : m_1'\equiv m_2' \bmod 3\}$ and
${{u_1^3=x^2y,u_2^3=xy^2}}$.  Thus, we get an expansion similar to
that of $f_{(1,1)}$ but the sum is only over the points in the first
quadrant that lie in the sublattice $L$ of index $3$ rather than all
of $\N^2$.

\end{example}

\end{document}